\documentclass[11pt]{article}
\usepackage{amsthm}
\usepackage{amsmath}
\usepackage{amssymb}
\usepackage{latexsym}
\usepackage{amsfonts}
\usepackage{placeins}
\usepackage{caption}
\usepackage{mathrsfs}
\usepackage{listings}
\newtheorem{theorem}{\bf Theorem}[section]
\newtheorem{proposition}[theorem]{\bf Proposition}

\newtheorem{lemma}[theorem]{\bf Lemma}
\begin{document}
\title{Sums of three Fibonacci numbers close to a power of 2}
\author{Bibhu Prasad Tripathy and Bijan Kumar Patel}
\date{}
\maketitle
\begin{abstract}
 In this paper, we find all the sums of three Fibonacci numbers which are close to a power of 2. This paper continues and extends the previous work of Hasanalizade \cite{Hasanalizade}.
\end{abstract} 

\section{Introduction}
The Fibonacci sequence $\{F_n\}_{n\geq 0}$ is the binary recurrence sequence  given by 
\[
F_{n+2} = F_{n+1} + F_n~~ {\rm for} ~n\geq 0
\]
with initials $F_0 = 0$ and $F_1 = 1$. The Fibonacci numbers are famous for possessing wonderful and amazing properties. The problem of determining all integer solutions to Diophantine equations with Fibonacci numbers has  piqued the curiosity of mathematicians and there is a wide literature on this subject. For instance, Bugeaud et al. \cite{Bugeaud} studied the Fibonacci numbers, that is, perfect powers and found that 1, 2 and 8 are the only powers of 2  in the Fibonacci sequence. Later, Bravo and Luca \cite{Bravo3} solved the Diophantine equation $F_{n} + F_{m} = 2^{a}$ in non-negative integers $n, m$ and $a$ with $n \geq m$. Following that Bravo and Bravo \cite{Bravo1} found all solutions of the Diophantine equation $F_{n} + F_{m}+ F_{l} = 2^{a}$ in non-negative integers $n, m, l$ and $a$ with $n \geq m \geq l$. Consequently, Chim and Ziegler \cite{Chim} determined all solutions of the Diophantine equations 
\[
F_{n_{1}} +F_{n_{2}} = 2^{a_{1}} + 2^{a_{2}} + 2^{a_{3}}
\]
and 
\[
F_{m_{1}} + F_{m_{2}} + F_{m_{3}} = 2^{t_{1}} + 2^{t_{2}}
\]
in non-negative integers $n_{1}, n_{2}, m_{1}, m_{2}, m_{3}, a_{1}, a_{2}, a_{3}, t_{1}, t_{2}$ with $n_{1} \geq n_{2}, m_{1} \geq m_{2} \geq m_{3}, a_{1} \geq a_{2} \geq a_{3}$ and $t_{1} \geq t_{2}$.

An integer $n$ is said to be close to a positive integer $m$ if it satisfies
\[
| n - m | < \sqrt{m}.
\]
By using the above concept Chern and Cui \cite{Chern} found all the Fibonacci numbers which are close to the power of 2. More precisely they found all  solutions of the inequality 
\[
|F_{n} - 2^{m}| < 2^{m/2}.
\]
Later, Bravo et al. \cite{Bravo2} extended the previous work of \cite{Chern} by considering the $k$-generalized Fibonacci sequence $\{F_{n}^{(k)}\}$ and solved the Diophantine equation
\[
|F_{n}^{(k)} - 2^{m} | < 2^{m/2}
\]
in  non-negative integer $n,k,m$ with $k \geq 2$ and $n \geq 1$. 

Recently, Hasanlizade \cite{Hasanalizade} extended the previous work of \cite{Chern} by considering the sum of Fibonacci numbers and studied the sum of two Fibonacci numbers close to a power of 2. In particular, he solved 
\[
|F_{n} + F_{m} - 2^{a}| < 2^{a/2},
\]
in positive integers $n, m$ and $a$ with $n \geq m$.

Motivated by the above literature, we extend the work of \cite{Chern} and \cite{Hasanalizade} search for the sum of three Fibonacci numbers which are close to the power of 2. More specifically, we study the Diophantine inequality 
\begin{equation}\label{eq 1.1}
    |F_{n} + F_{m} + F_{l} - 2^{a}| < 2^{a/2},
\end{equation}
in positive integers $n,m,l$ and $a$ with $n\geq m \geq l$.

In particular, our main result concerning \eqref{eq 1.1} is the following.
\begin{theorem}\label{thm1}
There are exactly 214 solutions $(n,m,l,a)$ to Diophantine inequality \eqref{eq 1.1}. All the solutions satisfy $n \leq 42$ and  $a \leq 28$. A list of solutions is given in the appendix.
\end{theorem}

%The Lucas sequence $(L_{n})_{n \geq 0}$ is a popular sequence that has been described as similar to the Fibonacci sequence. It has the same recursive relation as the Fibonacci number, but with initial condition $L_{0}=2$ and $L_{1}=1$. There is a very important relation between Fibonacci numbers and Lucas numbers which is defined as 
%\[
%L_{k} = F_{k-1} + F_{k+1} \quad \text{for all $k \geq 1$}.
%\]
%As a corollary here we determine all Lucas numbers, sum of two Lucas number and sum of three Lucas number which are close to power of 2.
%\begin{corollary}
%There are exactly $9$ Lucas numbers which are close to power of 2. Namely, the solutions $(n, a) \in \mathbb{N}^{2}$ of the inequality 
%\[
%|L_{n} - 2^{a}| < 2^{a/2}
%\] 
%are $(1,1), (2, 1), (2, 2), (3, 2), (4, 3), (6, 4), (7, 5), (10, 7) \text{and} \  (13, 9)$.

%\
%\end{corollary}
%\begin{corollary}
%There are exactly 34 solutions $(n, m, a) \in \mathbb{N}^{3}$ to the Diophantine inequality 
%\[
%|L_{n} + L_{m} - 2^{a}| < 2^{a/2}.
%\]
%All solutions satisfy $n \leq 36$ and $a \leq 25$. A list of solutions is given in Table 2 of the appendix.
%\end{corollary}
%\begin{corollary}
%There are exactly 151 solutions $(n, m, l, a) \in \mathbb{N}^{4}$ to the Diophantine inequality 
%\[
%|L_{n} + L_{m}+L_{l} - 2^{a}| < 2^{a/2}.
%\]
%All solutions satisfy $n \leq 36$ and $a \leq 25$. A list of solutions is given in Table 3 of the appendix.
%\end{corollary}

\section{Auxiliary results}
In this part, we will review several well-known results which will be used after. The Binet formula for Fibonacci sequence is 
\begin{equation}\label{eq 2.2}
    F_{n}=\frac{\alpha^{n}-\beta^{n}}{\alpha - \beta} \quad (n \geq 0),
\end{equation} 
where $\alpha  = \frac{1+\sqrt{5}}{2}$ and $\beta = \frac{-1}{\alpha}$ are the roots of the characteristic equation $x^2-x-1=0$.
\begin{lemma} \label{lm1}
For every positive integer $n \geq 1$, we have 
\begin{equation}\label{eq 2.3}
\alpha^{n-2} \leq F_{n} \leq \alpha^{n-1}. 
\end{equation}
\end{lemma}
\begin{lemma} (\cite[Eq. 5]{Ziegler}) \label{lm01}
For every positive integer $n > 1$, we have 
\begin{equation*}
0.38 \alpha^{n} \leq F_{n} \leq 0.48 \alpha^{n}. 
\end{equation*}
\end{lemma}

To prove our main result, we use a few times a Baker-type lower bound for  non-zero linear forms in logarithms of algebraic numbers. We state a result of Matveev \cite{Matveev} about the general lower bound for linear forms in logarithms. Applying a version of the Baker-Davenport reduction method we reduce the large bound, but first, recall some basic notations from algebraic number theory.

Let $\eta$ be an algebraic number of degree $d$ with minimal primitive polynomial 
\[
f(X):= a_0 X^d+a_1 X^{d-1}+ \cdots +a_d = a_0 \prod_{i=1}^{d}(X- \eta^{(i)}) \in \mathbb{Z}[X],
\]
where the $a_i$'s are relatively prime integers, $a_0 >0$, and the $\eta^{(i)}$'s are conjugates of $\eta$. Then
\begin{equation}\label{eq03}
h(\eta)=\frac{1}{d}\left(\log a_0+\sum_{i=1}^{d}\log\left(\max\{|\eta^{(i)}|,1\}\right)\right)
\end{equation}
is called the \emph{logarithmic height} of $\eta$. In particular, if $\eta = p/q$ is a rational number with gcd$(p, q)= 1$ and $q > 0$, then  $h(\eta)= \log \max \{|p|, q \}$.

%The following properties of the logarithms height, which will be used in the next section: 
%\[
 %h(\eta \pm \gamma ) \leq h(\eta) + h(\gamma) + \log2.
 %\]
%\[ 
 %h(\eta \gamma^{\pm} ) \leq h(\eta) + h(\gamma).
 %\]
%\[
%h(\eta) = h(\eta^{(i)}).
%\]
%\[
%h(\eta^{s}) = |s|h(\eta).
%\]

With these established notations, Matveev (see  \cite{Matveev} or  \cite[Theorem~9.4]{Bugeaud}), proved the ensuing result.

\begin{theorem}\label{thrm2}
  Let $\mathbb{K}$ be a number field of degree $D$ over $\mathbb{Q}$, $ \gamma_{1}, \gamma_{2}, \dots ,\gamma_{t}$ be positive real numbers of $\mathbb{K}$, and $b_{1}, . . . , b_{t}$ rational integers. Put
\[
B \geq max \{ |b_{1}|,|b_{2}|,\dots ,|b_{t}| \},
\]
and 
\[
\Lambda := \gamma_{1} ^ {b_{1}} \dots \gamma_{t} ^ {b_{t}} - 1.
\]
Let $ A_{1}, . . . , A_{t}$ be real numbers such that 
\[
A_{i} \geq max \{ D h(\gamma_{i}),|\log \gamma_{i} | ,0.16 \}, \quad i= 1,\dots,t.
\]
Then, assuming that $\Lambda \neq 0$, we have 
\[
|\Lambda| > exp \left( - 1.4 \times 30^{t+3} \times t^{4.5} \times D^2 (1 + \log D)(1 + log B ) A_{1} \cdots A_{t}     \right).
\]
\end{theorem}

The following criterion of Legendre, a well-known result from the theory of Diophantine approximation, is used to reduce the upper bounds on variables that are too large.

\begin{lemma}\label{lm2}
Let $\tau$ be an irrational number, $\frac{p_0}{q_0}, \frac{p_1}{q_1}, \frac{p_2}{q_2}, \dots$ be all the convergents of the continued fraction of $\tau$, and $M$ be a positive integer. Let $N$ be a non-negative integer such that $q_N > M$. Then putting $a(M) := \max \{a_i: i=0,1,2,\dots, N \}$, the inequality
\[
\Bigm| \tau - \frac{r}{s} \Bigm| > \frac{1}{(a(M)+2)s^2},
\]
holds for all pairs $(r, s)$ of positive integers with $0 < s < M$.
\end{lemma}

Another result which will play an important role in our proof is due to Dujella and Peth\"{o} \cite[Lemma~5 (a)]{Dujella}.

\begin{lemma} \label{lm3}
Let $M$ be a positive integer, let $p/q$ be a convergent of the continued fraction of the irrational $\gamma$ such that $q > 6M$, and let $A,B,\mu$ be some real numbers with $A>0$ and $B>1$. Let $\epsilon:=||\mu q||-M||\gamma q||$, where $||\cdot||$ denotes the distance from the nearest integer. If $\epsilon >0$, then there exists no solution to the inequality
\[
0< |u \gamma-v+\mu| <AB^{-w},
\]
in positive integers $u$, $v$ and $w$ with
\[
u \leq M \quad\text{and}\quad w \geq \frac{\log(Aq/\epsilon)}{\log B}.
\]
\end{lemma}

\begin{lemma}\label{lm4}
For any non-zero real number $x$, we have the following

(a) $0 < x < e^{x} - 1$.

(b) If $x < 0$ and $|e^{x} - 1| < 1/2$, then $|x| < 2|e^{x} - 1|$.
\end{lemma}

\section{Proof of Theorem \ref{thm1}}
\subsection{Upper bound for $n$.}
First of all observe that, $(n,m,l,a)$ is a solution of  \eqref{eq 1.1} for $m=l=1$ and $m=l=2$. So from now on, we assume that $m\geq 2$ and $l \geq2$. Since $F_{n} + F_{n-1} = F_{n+1}$, we can assume that $n> m+1$ and $n > l+1$. In particular $n-m \geq 2$ and $n - l \geq 2$.

\begin{proposition}\label{prop 3.1}
There are exactly $214$ solutions $(n, m, l, a) \in \mathbb{N}^{4}$ to \eqref{eq 1.1} with $n \leq 550$. All solutions satisfy $n \leq 42$ and $a \leq 28$. The list of solutions is given in the appendix. 
\end{proposition}
\begin{proof}
The solutions were found by a brute force search with \textit{Mathematica}.
\end{proof}
Due to Proposition \ref{prop 3.1}, for the rest of paper, we assume that $n > 550$. Using Lemma \ref{lm01} for $n > m > l > 1$, we have 
\begin{equation}\label{eq 3.6}
    0.38\alpha^{n} < F_{n} < F_{n} + F_{m} + F_{l} < 0.48\alpha^{n} + 0.48\alpha^{n-1} + 0.48\alpha^{n-2} < 0.97 \alpha^{n}.
\end{equation}

Let us now get a relationship between $n$ and $a$. Without affecting generality, we may now assume $a \geq 2$. Combining \eqref{eq 1.1} with the right inequality of \eqref{eq 3.6}, we get that
\begin{equation}\label{eq 3.7}
    2^{a-1} \leq 2^{a} - 2^{\frac{a}{2}} <  F_{n} + F_{m} + F_{l} < 0.97\alpha^{n} < \alpha^{n}.
\end{equation}
When we combine \eqref{eq 1.1} with the left inequality of \eqref{eq 3.6}, we get
\[
0.38 \alpha^{n} <  F_{n} + F_{m} + F_{l} <  2^{a} + 2^{\frac{a}{2}} < 2^{a+1}.
\]
Thus 
\begin{equation}\label{eq 3.8}
    n \frac{\log\alpha}{\log2} + \frac{\log0.38}{\log2} - 1 < a <  n \frac{\log\alpha}{\log2} + 1, 
\end{equation}
where $\frac{\log\alpha}{\log2} = 0.6942\dots$. Hence, we have $a < n$. Using \eqref{eq 2.2} in \eqref{eq 1.1}, we get
\begin{equation}\label{eq 3.9}
    \frac{\alpha^{n} + \alpha^{m} + \alpha^{l}}{\sqrt{5}} -  F_{n} + F_{m} + F_{l} = \frac{\beta^{n} + \beta^{m} + \beta^{l}}{\sqrt{5}}.
\end{equation}
Taking the absolute values on both sides in \eqref{eq 3.9}, we obtain 
\[
\left|\frac{\alpha^{n}(1 + \alpha^{m-n} + \alpha^{l-n})}{\sqrt{5}} - 2^{a} \right| < \frac{|\beta|^{n} + |\beta|^{m} + |\beta|^{l}}{\sqrt{5}} + 2^{\frac{a}{2}} < \frac{3}{5} + 2^{\frac{a}{2}} < 2^{\frac{a}{2}+1},
\]
for all $n \geq 4$ and $m, l \geq 2$.
\\
Dividing both sides of the above relation by $2^{a}$, we get 
\begin{equation}\label{eq 3.10}
    \left| \frac{\alpha^{n}(1 + \alpha^{m-n} + \alpha^{l-n})}{2^{a}\sqrt{5}} - 1 \right| < 2^{-\frac{a}{2}+1}.
\end{equation}
For the left-hand side, we apply Theorem \ref{thrm2} with the following data. Set
\[
\Lambda_{1} :=  2^{-a}\alpha^{n}\frac{(1 + \alpha^{m-n} + \alpha^{l-n})}{\sqrt{5}} - 1.
\]
Note that $\Lambda_{1} \neq 0$. If  $\Lambda_{1} = 0$, we would get the relation 
\begin{equation}\label{eq 3.11}
    2^{a}\sqrt{5} = \alpha^{n} + \alpha^{m} + \alpha^{l}.
\end{equation}
Conjugating the above relation in $\mathbb{Q}(\sqrt{5})$, we get
\begin{equation}\label{eq 3.12}
    -2^{a}\sqrt{5} = \beta^{n} + \beta^{m} + \beta^{l}.
\end{equation}
Equations \eqref{eq 3.11} and \eqref{eq 3.12} lead to
\begin{equation*}
\alpha^{n} < \alpha^{n} + \alpha^{m} + \alpha^{l} = | \beta^{n} + \beta^{m} + \beta^{l} | < | \beta^{n}| + |\beta^{m}| +  |\beta^{l} | < 1,
\end{equation*}
which is impossible for positive integer $n$. Hence $\Lambda_{1} \neq 0$. We take $t := 3$, 
\[
\gamma_1 := 2, \quad \gamma_2 := \alpha , \quad \gamma_3 := \frac{(1 + \alpha^{m-n} + \alpha^{l-n})}{\sqrt{5}},
\]
and
\[ b_{1}:= -a, \quad b_{2}:= n, \quad b_{3}:= 1.
\]
\\ 
Note that $\mathbb{K} := \mathbb{Q}(\sqrt{5})$ contains $\gamma_{1}, \gamma_{2},  \gamma_{3}$ and has $D:= [\mathbb{K}: \mathbb{Q}] = 2$. Since $a < n$, we can take $B:= \max\{|B_{1}|, |B_{2}|, |B_{3}|\} = n$.
%The three numbers $\gamma_{1}, \gamma_{2}, \ \text{and} \  \gamma_{3}$ are positive real numbers and belong to $\mathbb{K} = \mathbb{Q}(\sqrt{5})$, so we can take $D:= [\mathbb{K}: \mathbb{Q}] = 2$. By recalling that $a < n$, we can take $B:= \max\{|B_{1}|, |B_{2}|, |B_{3}|\} = n$. 
Since $h(\gamma_{1}) = \log2$ and $h(\gamma_{2}) = \frac{\log\alpha}{2}$, it follows that we can take $A_{1}:= 1.4 > Dh(\gamma_{1})$ and $A_{2}:= 0.5 > Dh(\gamma_{2})$. To estimate $h(\gamma_{3})$, we begin by observing that 
\[
\gamma_{3} = \frac{(1 + \alpha^{m-n} + \alpha^{l-n})}{\sqrt{5}} < \frac{3}{\sqrt{5}} \quad \text{and} \quad \gamma_{3}^{-1} = \frac{\sqrt{5}}{(1 + \alpha^{m-n} + \alpha^{l-n})} < \sqrt{5},
\]
it follows that $|\log\gamma_{3}| < 1$. Using the properties of logarithmic height, we get
\begin{align*}
 h(\gamma_{3}) & \leq \log\sqrt{5} + |m-n|\left(\frac{\log\alpha}{2}\right) + |l-n|\left(\frac{\log\alpha}{2}\right) + 2\log2 \\ 
 &= \log(4\sqrt{5}) + (n - m)\left(\frac{\log\alpha}{2}\right) + (n-l)\left(\frac{\log\alpha}{2}\right).   
\end{align*}
Thus, we can take $A_{3} := 5 + (2n-m-l)\log\alpha > \max \{2h(\gamma_{3}), |\log\gamma_{3}|, 0.16\}$. Then by using Theorem \ref{thrm2}, we have
\begin{align*}
 \log|\Lambda_{1}| &> -1.4 \times 30^{6} \times 3^{4.5} \times 2^{2} \times(1+ \log2) \times (1 +\log n) \times1.4 \times0.5 \times ( 5 + (2n-m-l)\log \alpha) \\
 & > -1.4 \times 10^{12} \times \log n \times ( 5+ (2n-m-l)\log \alpha),
\end{align*}
where $1 + \log n < 2 \log n $ holds for all $n \geq 3$. By comparing the above inequality with \eqref{eq 3.10}, we get 
\begin{equation}\label{eq 3.13}
    \left( \frac{a}{2} -1 \right) \log 2 < 1.4 \times 10^{12} \times \log n \times ( 5 + (2n-m-l)\log \alpha).
\end{equation}
Now consider the second linear form in logarithms by rewriting \eqref{eq 3.9} differently. Using \eqref{eq 2.2}, we get that 
\[
 \frac{\alpha^{n}}{\sqrt{5}} -  (F_{n} + F_{m} + F_{l}) =  \frac{\beta^{n}}{\sqrt{5}} -  F_{m} + F_{l}.
 \]
 Again, combining the above relation with \eqref{eq 1.1}, we get that 
 \begin{equation}\label{eq3.15}
 \left| \frac{\alpha^{n}}{\sqrt{5}} - 2^{a} \right| < 2^{\frac{a}{2}} + \frac{|\beta|^{n}}{\sqrt{5}} + F_{m} + F_{l} ,  \end{equation}
where $|\beta|^{n} < 1/2$ for all $n \geq 2$. Dividing both sides of inequality \eqref{eq3.15} by $\frac{\alpha^{n}}{\sqrt{5}}$ and taking account that $\alpha > \sqrt{2}$ and $n> m > l$, we obtain 
\begin{align}\label{eq 3.14}
  | 1 - 2^{a}\alpha^{-n} \sqrt{5}| &< \frac{2^{\frac{a}{2}}\sqrt{5}}{\alpha^{n}} +  \frac{\sqrt{5}}{2\alpha^{n}} + \frac{\sqrt{5}}{\alpha^{n-m}} + \frac{\sqrt{5}}{\alpha^{n-l}} \nonumber \\
  &< 2\sqrt{5}\max\{\alpha^{m-n},\alpha^{l-n},\alpha^{a-n}\}.  
\end{align}
For the left-hand side, we apply Theorem \ref{thrm2} with the following data. Set
\[
\Lambda_{2} := 2^{a}\alpha^{-n} \sqrt{5} - 1.
\]
Note that $\Lambda_{2} \neq 0$. If  $\Lambda_{2} = 0$, then we have that $2^{a} = \frac{\alpha^{n}}{\sqrt{5}}$. So $\alpha^{2n} \in \mathbb{Z}$, which is not possible. Therefore $\Lambda_{2} \neq 0$. We take $t := 3$, 
\[
 \gamma_{1}:= 2, \quad \gamma_{2}:= \alpha, \quad \gamma_{3}:= \sqrt{5},
\]
and
\[ 
b_{1}:= a, \quad b_{2}:= -n, \quad b_{3}:= 1.
\]
Note that $\mathbb{K} := \mathbb{Q}(\sqrt{5})$ contains $\gamma_{1}, \gamma_{2}, \gamma_{3}$ and has $D:= 2$. Since $a < n$, we deduce that $B:= \max\{|b_{1}|,|b_{2}|,|b_{3}|\}= n.$ 
%The three numbers $\gamma_{1}, \gamma_{2}, \ \text{and} \ \gamma_{3}$ are positive real numbers and belong to $\mathbb{K} := \mathbb{Q}(\sqrt{5})$, so we can take $D:= [\mathbb{K}: \mathbb{Q}] = 2$. Since $a < n$, we deduce that $B:= \max\{|b_{1}|,|b_{2}|,|b_{3}|\}= n.$ 
%We apply again Theorem \ref{thrm2} with same $\mathbb{K} = \mathbb{Q}(\sqrt{5})$ as before. We take $t = 3$, 
%\[
 %\gamma_{1}:= 2, \quad \gamma_{2}:= \alpha, \quad \gamma_{3}:= \sqrt{5},
%\]
%and
%\[ b_{1}:= a, \quad b_{2}:= -n, \quad b_{3}:= 1.
%\]
%Hence,
%\[
%\Lambda_{2} := 2^{a}\alpha^{-n} \sqrt{5} - 1.
%\]
%Since three numbers $ \gamma_{1},  \gamma_{2},  \gamma_{3}$ are real, positive and belongs to $\mathbb{K}$, we can take $D=2$. Here $\Lambda_{2} \neq 0$. If $\Lambda_{2} = 0$, then we have that $2^{a} = \frac{\alpha^{n}}{\sqrt{5}}$. So $\alpha^{2n} \in \mathbb{Z}$, which is impossible.
The logarithmic heights for $\gamma_{1}$, $\gamma_{2}$ and $\gamma_{3}$ are calculated as follows:
\[
h(\gamma_{1}) = \log2, \quad h(\gamma_{2}) = \frac{\log\alpha}{2} \quad \text{and} \quad h(\gamma_{3}) = \log\sqrt{5}.
\]
Thus, we can take
\[
A_{1} := 1.4, \quad A_{2} := 0.5 \quad \text{and} \quad A_{3} := 1.7.
\]
As before, applying Theorem \ref{thrm2}, we have 
\begin{align*}
 \log|\Lambda_{2}| &> -1.4 \times 30^{6} \times 3^{4.5} \times 2^{2} \times(1+ \log2) \times (1 +\log n) \times1.4 \times0.5 \times 1.7 \\
 & > -2.31 \times 10^{12} \times \log n. 
\end{align*}
Comparing the above inequality with \eqref{eq 3.14} implies that 
\[
\min\{(n-m)\log\alpha, (n-l)\log\alpha, (n-a)\log\alpha  \} < 2.4 \times 10^{12} \log n.
\]

%Since $h(\gamma_{1}) = \log2 = 0.6931\dots$,  we can choose $A_{1}:= 1.4 > Dh(\gamma_{1})$. Further  $h(\gamma_{2}) = \frac{\log\alpha}{2} = 0.2406\dots$ and $h(\gamma_{3}) = \log\sqrt{5}= 0.8047\dots$, it follows that we can take  $A_{2}:= 0.5 > Dh(\gamma_{2})$ and $A_{3}:= 1.7 > Dh(\gamma_{3})$. Finally since $a < n$, we deduce that $B:= \max\{|b_{1}|,|b_{2}|,|b_{3}|\}= n.$

%As before, by  Theorem \ref{thrm2}, we have 
%\begin{align*}
 %\log|\Lambda_{2}| &> -1.4 \times 30^{6} \times 3^{4.5} \times 2^{2} \times(1+ \log2) \times (1 +\log n) \times1.4 \times0.5 \times 1.7 \\
% & > -2.31 \times 10^{12} \times \log n. 
%\end{align*}
%Comparing  the above inequality with \eqref{eq 3.14}, we get 
%\[
%\min\{(n-m)\log\alpha, (n-l)\log\alpha, (n-a)\log\alpha  \} < 2.4 \times 10^{12} \log n.
%\]
Now the argument splits into three cases.  \\
\noindent
\textbf{Case 1.}
$\min\{(n-m)\log\alpha, (n-l)\log\alpha, (n-a)\log\alpha  \} = (n-m)\log\alpha$.
\\
\noindent
In this case, we have  
\[
(n-m)\log\alpha < 2.4 \times 10^{12} \log n.
\]
\textbf{Case 2.} $\min\{(n-m)\log\alpha, (n-l)\log\alpha, (n-a)\log\alpha  \} = (n-l)\log\alpha.$
\\
In this case, we have  
\[
(n-l)\log\alpha < 2.4 \times 10^{12} \log n.
\]
Now by using Case 1 and Case 2 in \eqref{eq 3.13}, we have 
\begin{align*}
\left( \frac{a}{2} -1 \right) \log 2 &< 1.4 \times 10^{12} \times \log n \times ( 5 + (2n-m-l)\log \alpha)  \\
& < 1.4 \times 10^{12} \times \log n \times ( 5 + 4.8 \times 10^{12} \log n)\\
& < 6.86 \times 10^{24} \log^{2} n.
\end{align*}
Combining it with the left inequality of \eqref{eq 3.8} and by a calculation in \textit{Mathematica}, we get 
\[
a < 9 \times10^{28} \quad  \text{and} \quad n< 1.87 \times10^{29}.
\]
\textbf{Case 3.} $\min\{(n-m)\log\alpha, (n-l)\log\alpha, (n-a)\log\alpha  \} = (n-a)\log\alpha.$ 
\\
In this case, we have  
\begin{equation}\label{eq3.16}
(n-a)\log\alpha < 2.4 \times 10^{12} \log n.    
\end{equation}
Note that the right inequality of \eqref{eq 3.8} yields
\begin{equation}\label{eq3.17}
n-a > n\left( 1 - \frac{\log\alpha}{\log2} \right) - 1.    
\end{equation}
From \eqref{eq3.16} and  \eqref{eq3.17}, we obtain
\[
a < 3.84 \times10^{14} \quad  \text{and} \quad n< 5.53 \times10^{14}.
\]
Thus, in all the three cases, we have 
\[
a < 9 \times10^{28} \quad  \text{and} \quad n< 1.87 \times10^{29}.
\]
Now we need to reduce the bound of $n$.
\subsection{Reducing the bound on $n$}  Let us assume that $n-m > 550$, $n-l > 550$  and $n-a > 550$. In order to apply Lemma \ref{lm3}, we put 
\[
z_{1}:= a\log2 - n \log \alpha + \log \sqrt{5}.
\]
Since we have assumed that $\min\{n-m, n-l, n-a\} > 550$ therefore $|e^{z_{1}}-1| < 1/2$. Thus, by Lemma \ref{lm4}, we have that $|z_{1}| < 2|e^{z_{1}} - 1|$. Replacing $z_{1}$ by its formula and by \eqref{eq 3.14}, we get that
\[
|a\log2 - n\log\alpha + \log \sqrt{5}| < \frac{4\sqrt{5}}{\alpha^{\min\{n-m, n-l, n-a\}}}.
\]
Dividing both sides of the above inequality by $\log \alpha$, we can conclude
\begin{equation}\label{eq 3.15}
0< \left| a\frac{\log2}{\log \alpha} - n + \frac{\log \sqrt{5}}{\log \alpha} \right| < \frac{4 \sqrt{5}}{\log \alpha} \cdot \alpha^{-w},    
\end{equation}
where $w = \min\{n-m, n-l, n-a\}$.
Putting now
\[
\gamma:= \frac{\log2}{\log \alpha}, \quad \mu:= \frac{\log \sqrt{5}}{\log \alpha}, \quad A:= \frac{4 \sqrt{5}}{\log \alpha} \quad \text{and} \quad B:= \alpha,
\]
the above inequality \eqref{eq 3.15} implies
\[
0 < |a\gamma - n + \mu| < AB^{-w}.
\]
It is clear that  $\gamma$ is an irrational number. We also put $M := 9\times10^{28} $ which is an upper bound for $a$. We find that the convergent $\frac{p}{q} = \frac{p_{69}}{q_{69}}$ is such that $q = q_{69} > 6M$. By using this we have that $\epsilon = ||\mu q_{69}||-M||\gamma q_{69}|| > 0$. Therefore 
\[
  n-m < \frac{\log \left( \frac{4\sqrt{5}q}{\epsilon \ \log \alpha}\right)}{\log \alpha} < 157 \quad  \text{or} \quad n-l < \frac{\log \left( \frac{4\sqrt{5}q}{\epsilon \ \log \alpha}\right)}{\log \alpha} < 157 \quad
\text{or} \quad n-a < \frac{\log \left( \frac{4\sqrt{5}q}{\epsilon \ \log \alpha}\right)}{\log \alpha} < 157.  
\]
Thus, we have that either $n-m < 157 $ or $n < 517$. The latter case contradicts our assumption that $n > 550$. Inserting the upper bound for $n-m$ into \eqref{eq 3.13}, we get that  $a < 3.93 \times 10^{15}$.

Finally, we shall use  inequality \eqref{eq 3.10} to improve the bound on $n$, where we put 
\[
z_{2}:= a \log 2 - n\log \alpha + \log \psi(n-m, n-l),
\]
where $\psi(t,s) := \sqrt{5}(1 + \alpha^{-t} + \alpha^{-s})^{-1}$. Therefore, \eqref{eq 3.10} implies that 
\begin{equation}\label{eq 3.16}
    |1 - e^{z_{2}}| < \frac{2}{2^{a/2}}.
\end{equation}
Note that $z_{2}\neq 0$. Thus, we distinguish the following cases. If $z_{2} > 0$, then 
\[
0 < z_{2} < e^{z_{2}} - 1 \leq \frac{2}{2^{a/2}}.
\]
Suppose $z_{2} < 0$. %It is a straight forward exercise to check that $\frac{2}{2^{a/2}} < 1/2$ for all $a > 28$. 
Then, from \eqref{eq 3.16} we have that $|1-e^{z_{2}}|< 1/2$ and therefore $e^{|z_{2}|} < 2.$ Since $z_{2} < 0$, we have 
\[
0 < |z_{2}| \leq e^{|z_{2}|} - 1 = e^{|z_{2}|}|e^{z_{2}} - 1 | < \frac{4}{2^{a/2}}.
\]
In any case, we have that  
\[
0 < |z_{2}| \leq \frac{4}{2^{a/2}}.
\]
Replacing $z_{2}$ in the above inequality and dividing both sides by $\log \alpha$, we conclude that 
\begin{equation}\label{eq 3.17}
  0 < \left| a \frac{\log 2}{\log \alpha} -n  + \frac{\log \psi(n-m, n-l)}{\log \alpha} \right| < \frac{4}{\log \alpha} \cdot 2^{-a/2}.
\end{equation}
Here we take $M:= 3.93 \times 10^{15}$ which is an upper bound for $a$. By virtue of Lemma \ref{lm3} with  
\[
\gamma:= \frac{\log2}{\log \alpha}, \quad \mu:= \frac{\log \psi(n-m, n-l)}{\log \alpha}, \quad A:= \frac{4}{\log \alpha}, B:= \sqrt{2} 
\]
to inequality \eqref{eq 3.17} for all choices $n-m \in \{1,\dots,157\}$ and  $n-l \in \{1,\dots,157\}$ except when 
\[
(n-m, n-l) \in \{ (0, 3), (1, 1), (1, 5), (3, 0), (3, 4), (4, 3), (5, 1), (7, 8), (8, 7) \}.
\]
Furthermore, with the help of \textit{Mathematica}, we find that if $(n, m, l, a)$ is a possible solution of \eqref{eq 1.1}, excluding those cases mentioned previously, then $n \leq 550$.
This is false because we have assumed $n > 550$.
\\
Finally, we deal with the special cases where
\[
(n-m, n-l) \in \{ (0, 3), (1, 1), (1, 5), (3, 0), (3, 4), (4, 3), (5, 1), (7, 8), (8, 7) \}.
\]
As we are seeing, there are certain symmetric cases and there is no significant difference in their study. Therefore, we deal with the cases $(n-m, n-l) \in \{(1, 1),(3, 0),(4, 3),(5, 1),(8, 7)\}.$ For these special cases, we have that 
\[
\frac{\log \psi(t,s) }{\log \alpha} = 
\begin{cases}
0 & \text{if}~~ (t, s) = (1, 1); \\
0 & \text{if}~~ (t, s) = (3, 0); \\
1 & \text{if}~~(t, s) = (4, 3); \\
2 - \frac{\log2}{\log \alpha} & \text{if}~~ (t, s) = (5, 1); \\
3 - \frac{\log2}{\log \alpha} & \text{if}~~ (t, s) = (8, 7).
\end{cases}
\]
Note that, when we apply Lemma \ref{lm3} to the inequality \eqref{eq 3.17}, the parameters $\gamma$ and $\mu$ appearing in Lemma \ref{lm3} are linearly dependent and the corresponding value of $\epsilon$ is always negative. When $n-m=1$ and $n-l=1$ from \eqref{eq 3.17}, we have 
\begin{equation}\label{eq 3.18}
    0 < |a\gamma - n | < \frac{4}{\log \alpha} \cdot 2^ {-\frac{1}{2}(n\frac{\log\alpha}{\log2}+ \frac{\log 0.38}{\log2}-1)}.
\end{equation}
We recall that $a < 3.93 \times 10^{15}$. Let $[ a_{0}, a_{1}, a_{2}, a_{3}, a_{4}, \dots] = [1, 2, 3, 1, 2, \dots]$ be the continued fraction of $\gamma$. A quick search using \textit{Mathematica} reveals that 
\[
q_{34} < 3.93\times 10^{15} < q_{35}.
\]
Furthermore, $a_{M}:= \max\{a_{i} : i=0, 1, \dots,35\} = a_{17} =134$. So by Lemma \ref{lm2} we have
\[
|a\gamma - n | > \frac{1}{(a_{M} + 2)a}.
\]
Comparing estimates \eqref{eq 3.17} and \eqref{eq 3.18} we get that $n < 112$. Similarly, one can get $n < 112$ in all other cases. This again contradicts our assumption that $n > 550$.  Hence, the result is proved.

\section{Appendix}
The solutions of Diophantine inequality \eqref{eq 1.1} are given in Table 1.

\begin{table} [ht]
\caption*{Table 1}
\centering
\begin{tabular}{ | p{7cm} | p {7cm} | }
    \hline
	  $|F_{2} + F_{2} + F_{2} - 2| < \sqrt{2}$ &  $|F_{11} + F_{8} + F_{i} - 2^{7}| < 2^{7/2}, i = 6, 7, 8$ \\
	\hline
	 $ |F_{2} + F_{2} + F_{2} - 2^{2}| < 2$ & $|F_{11} + F_{9} + F_{i} - 2^{7}| < 2^{7/2}$, $i = 2, 3,\dots, 8$   \\
	\hline
	$|F_{3} + F_{i} + F_{2} - 2^{2}| < 2, i=2, 3 $ & $|F_{11} + F_{11} + F_{11} - 2^{8}| < 2^{4}$  \\
	\hline
	
	 $|F_{4} + F_{2} + F_{2} - 2^{2}| < 2$ &  $|F_{12} + F_{10} + F_{10} - 2^{8}| < 2^{4}$ \\
	\hline
	 $|F_{4} + F_{3} + F_{i} - 2^{3}| < 2^{3/2}, i = 2, 3$ & $|F_{12} + F_{11} + F_{i} - 2^{8}| < 2^{4}, i = 6, 7, 8, 9$ \\
	\hline
	 $|F_{4} + F_{4} + F_{i} - 2^{3}| < 2^{3/2}, i=2, 3, 4$ & $|F_{13} + F_{5} + F_{i} - 2^{8}| < 2^{4}, i = 4, 5$  \\
	\hline
	 	$|F_{5} + F_{2} + F_{2} - 2^{3}| < 2^{3/2}$ & $|F_{13} + F_{6} + F_{i} - 2^{8}| < 2^{4}, i = 2, 3, 4, 5, 6$   \\
	\hline
	 $|F_{5} + F_{3} + F_{i} - 2^{3}| < 2^{3/2}, i = 2, 3$  & $|F_{13} + F_{7} + F_{i} - 2^{8}| < 2^{4}, i = 2, 3, 4, 5, 6, 7$   \\
	\hline
	$|F_{5} + F_{4} + F_{i} - 2^{3}| < 2^{3/2}, i=2, 3$ & $|F_{13} + F_{8} + F_{i} - 2^{8}| < 2^{4}, i = 2, 3, 4, 5, 6, 7$  \\
	\hline
	$|F_{6} + F_{2} + F_{2} - 2^{3}| < 2^{3/2}$ & $|F_{13} + F_{9} + F_{i} - 2^{8}| < 2^{4}, i = 2, 3, 4$   \\
	\hline
	$|F_{5} + F_{5} + F_{i} - 2^{4}| < 2^{2}, i=4, 5$ & $|F_{13} + F_{12} + F_{12} - 2^{9}| < 2^{9/2}$   \\
	\hline
	$|F_{6} + F_{4} + F_{i} - 2^{4}| < 2^{2}, i = 3, 4$ & $|F_{13} + F_{13} + F_{i} - 2^{9}| < 2^{9/2}, i = 9, 10$ \\
	\hline
 $|F_{6} + F_{5} + F_{i} - 2^{4}| < 2^{2}, i = 2, 3, 4, 5$ & $|F_{14} + F_{11} + F_{i} - 2^{9}| < 2^{9/2}, i = 9, 10$  \\
	\hline
$|F_{6} + F_{6} + F_{i} - 2^{4}| < 2^{2}, i = 2, 3, 4$	& $|F_{14} + F_{12} + F_{i} - 2^{9}| < 2^{9/2}, i = 2, 3, 4, 5, 6, 7$   \\
	\hline
$|F_{7} + F_{2} + F_{2} - 2^{4}| < 2^{2}$	 & $|F_{15} + F_{14} + F_{i} - 2^{10}| < 2^{5}, i = 6, 7, 8, 9, 10$  \\
	\hline
 $|F_{7} + F_{3} + F_{i} - 2^{4}| < 2^{2}, i = 2, 3$ & $|F_{16} + F_{4} + F_{4} - 2^{10}| < 2^{5}$  \\
	\hline
  $|F_{7} + F_{4} + F_{i} - 2^{4}| < 2^{2}, i = 2, 3, 4$  & $|F_{16} + F_{5} + F_{i} - 2^{10}| < 2^{5}, i = 2, 3, 4, 5$ \\
	\hline
	$|F_{7} + F_{5} + F_{2} - 2^{4}| < 2^{2}$    & $|F_{16} + F_{6} + F_{i} - 2^{10}| < 2^{5}, i = 2, 3, 4, 5, 6$    \\
	\hline
$|F_{7} + F_{6} + F_{6} - 2^{5}| < 2^{5/2}$	& $|F_{16} + F_{7} + F_{i} - 2^{10}| < 2^{5}, i = 2, 3, 4, 5, 6, 7$   \\
	\hline
$|F_{7} + F_{7} + F_{i} - 2^{5}| < 2^{5/2}, i = 2, 3, 4, 5, 6$	  & $|F_{16} + F_{8} + F_{i} - 2^{10}| < 2^{5}, i= 2, 3, 4,\dots, 8$   \\
	\hline
$|F_{8} + F_{4} + F_{4} - 2^{5}| < 2^{5/2}$	 & $|F_{16} + F_{9} + F_{i} - 2^{10}| < 2^{5}, i = 2, 3,\dots, 9$    \\
	\hline
	$|F_{8} + F_{5} + F_{i} - 2^{5}| < 2^{5/2}, i = 2, 3, 4, 5$ & $|F_{16} + F_{10} + F_{i} - 2^{10}| < 2^{5}, i = 2, 3, 4, 5, 6, 7$  \\
	\hline
	 $|F_{8} + F_{6} + F_{i} - 2^{5}| < 2^{5/2}, i = 2, 3, 4, 5, 6$  & $|F_{16} + F_{16} + F_{i} - 2^{10}| < 2^{5}, i = 9, 10, 11, 13$   \\
	\hline
	  $|F_{8} + F_{7} + F_{i} - 2^{5}| < 2^{5/2}, i = 2, 3, 4$   &  $|F_{17} + F_{13} + F_{13} - 2^{11}| < 2^{11/2}$  \\
	\hline
	$|F_{9} + F_{i} + F_{2} - 2^{5}| < 2^{5/2}, i = 2, 3$ & $|F_{17} + F_{14} + F_{i} - 2^{11}| < 2^{11/2}, i = 9, 10, 11$ \\
	\hline
  
 $|F_{8} + F_{8} + F_{8} - 2^{6}| < 2^{3}$ & $|F_{20} + F_{16} + F_{14} - 2^{13}| < 2^{13/2}$   \\
   \hline
  $|F_{9} + F_{7} + F_{7} - 2^{6}| < 2^{3}$  & $|F_{21} + F_{21} + F_{21} - 2^{15}| < 2^{15/2}$  \\
   \hline
 $|F_{9} + F_{8} + F_{i} - 2^{6}| < 2^{3}, i = 3, 4, 5, 6, 7 $ & $|F_{22} + F_{21} + F_{19} - 2^{15}| < 2^{15/2}$   \\
	\hline
 $|F_{9} + F_{9} + F_{i} - 2^{6}| < 2^{3}, i = 2, 3, 4 $ & $|F_{23} + F_{18} + F_{17} - 2^{15}| < 2^{15/2}$  \\
	\hline
$|F_{10} + F_{i} + F_{2} - 2^{6}| < 2^{3}, i = 2, 3$	 & $|F_{23} +F_{19} + F_{i} - 2^{15}| < 2^{15/2}$,$i= 2, 3,\dots, 11$  \\
   \hline
  $|F_{10} + F_{i} + F_{3} - 2^{6}| < 2^{3}, i = 3, 4$  & $|F_{24} + F_{22} + F_{17} - 2^{16}| < 2^{8}$   \\
	\hline
$|F_{10} + F_{4} + F_{i} - 2^{6}| < 2^{3}, i = 2, 4$  & $|F_{26} + F_{20} + F_{18} - 2^{17}| < 2^{17/2}$  \\
   \hline
$|F_{10} + F_{5} + F_{i} - 2^{6}| < 2^{3}, i = 2, 3, 4, 5$  & $|F_{29} + F_{20} + F_{18} - 2^{19}| < 2^{19/2}$   \\
   \hline
   $|F_{10} + F_{6} + F_{i} - 2^{6}| < 2^{3}, i = 2, 3, 4, 5, 6$ & $|F_{41} + F_{40} + F_{29} - 2^{28}| < 2^{14}$  \\
   \hline
$|F_{10} + F_{7} + F_{i} - 2^{6}| < 2^{3}, i = 2, 3, 4$     & $|F_{42} + F_{28} + F_{27} - 2^{28}| < 2^{14}$ \\
	\hline
	$|F_{10} + F_{9} + F_{9} - 2^{7}| < 2^{7/2}$   & $|F_{42} + F_{29} + F_{i} - 2^{28}| < 2^{14}$,$i= 2, 3,\dots, 22$ \\
   \hline
  $|F_{10} + F_{10} + F_{i} - 2^{7}| < 2^{7/2}, i = 6, 7, 8$  &   \\
	\hline
\end{tabular}
\end{table}

\vspace{05mm} \noindent \footnotesize
\begin{minipage}[b]{90cm}
\large{Department of Mathematics, \\ 
KIIT University, Bhubaneswar, \\ 
Bhubaneswar 751024, Odisha, India. \\
Email: bptbibhu@gmail.com}
\end{minipage}

\vspace{05mm} \noindent \footnotesize
\begin{minipage}[b]{90cm}
\large{Department of Mathematics, \\ 
KIIT University, Bhubaneswar, \\ 
Bhubaneswar 751024, Odisha, India. \\
Email: bijan.patelfma@kiit.ac.in}
\end{minipage}

\end{document}